\documentclass[12pt,oneside]{amsart}

      \usepackage{amssymb}
\usepackage[margin=1in]{geometry}  

\usepackage{ulem}       
\usepackage{amsmath, amsthm, amssymb}
\usepackage{xspace}
\usepackage{mathrsfs}
 \usepackage{graphicx}  
 \usepackage{color}	
 \usepackage{mathabx}
\newtheorem{theorem}{Theorem}[section]

\newtheorem{lemma}[theorem]{Lemma}

\newtheorem{example}[theorem]{Example}
 
\newcommand{\scr}{\mathscr}   
   
      \makeatletter
      \def\@setcopyright{}
      \def\serieslogo@{}
      \makeatother
 
\begin{document}

   \author{Amin  Bahmanian}
   \address{Department of Mathematics and Statistics
 Auburn University, Auburn, AL USA   36849-5310}
   \email{mzb0004@tigermail.auburn.edu}


   \title[Detachments of Hypergraphs]{Detachments of Hypergraphs I: The Berge-Johnson Problem}

   \begin{abstract}
A detachment of a hypergraph is formed by splitting each vertex into one or more subvertices, and sharing the incident edges arbitrarily among the subvertices. For a given edge-colored hypergraph $\scr F$, we prove that there exists a detachment  $\scr G$ such that the degree of each vertex and the multiplicity of each edge in $\scr F$ (and each color class of $\scr F$) are shared fairly among the subvertices in $\scr G$ (and each color class of $\scr G$, respectively).

Let $(\lambda_1\dots,\lambda_m) K^{h_1,\dots,h_m}_{p_1,\dots,p_n}$ be a hypergraph with vertex partition $\{V_1,\dots, V_n\}$, $|V_i|=p_i$ for $1\leq i\leq n$ such that there are $\lambda_i$ edges of size $h_i$ incident with every $h_i$ vertices, at most one vertex from each part for $1\leq i\leq m$ (so no edge is incident with more than one vertex of a part).  
We use our detachment theorem to show that the obvious necessary conditions for $(\lambda_1\dots,\lambda_m) K^{h_1,\dots,h_m}_{p_1,\dots,p_n}$ to be  
  expressed as the union $\scr G_1\cup \ldots \cup\scr G_k$ of $k$ edge-disjoint factors, where for $1\leq i\leq k$, $\scr G_i$ is $r_i$-regular,  are also sufficient.
Baranyai solved the case of $h_1=\dots=h_m$, $\lambda_1=\dots,\lambda_m=1$, $p_1=\dots=p_m$, $r_1=\dots =r_k$. Berge and  Johnson, (and later Brouwer and Tijdeman, respectively) considered (and solved, respectively) the case of $h_i=i$, $1\leq i\leq m$, $p_1=\dots=p_m=\lambda_1=\dots=\lambda_m=r_1=\dots =r_k=1$. We also extend our result to the case where each  $\scr G_i$ is almost regular.
   \end{abstract}


   \date{\today}

   \maketitle

\section {Introduction}  
Intuitively speaking, a  detachment of a hypergraph is formed by splitting each vertex into one or more subvertices, and sharing the incident edges arbitrarily among the subvertices. As the main result of this paper (see Theorem \ref{mainthhypgen1}), we prove that for a given edge-colored hypergraph $\scr F$,  there exists a detachment  $\scr G$ such that the degree of each vertex and the multiplicity of each edge in $\scr F$ (and each color class of $\scr F$) are shared fairly among the subvertices in $\scr G$ (and each color class of $\scr G$, respectively). This result is not only interesting by itself and generalizes various graph theoretic results (see for example \cite{BahRod1, H2, HR, MatJohns, LR1, LR3, Nash87, RW}), but also is used to obtain extensions of existing results on edge-decompositions of hypergraphs by Bermond, Baranyai \cite{Baran75, Baran79}, Berge and Johnson \cite{BergeJohnson75, ELJohnson78}, and Brouwer and Tijdeman \cite{Brouwer77hered, BrouTijd81}.

Given a set $N$ of $n$ elements, Berge and Johnson \cite{BergeJohnson75, ELJohnson78} addressed the question of when do there exist disjoint partitions of $N$, each partition containing only  subsets of $h$ or fewer elements, such that every subset of $N$ having $h$ or fewer elements is in exactly one partition. Here we state the problem in a more general setting with the hypergraph theoretic notation.

Let $(\lambda_1\dots,\lambda_m) K^{h_1,\dots,h_m}_{p_1,\dots,p_n}$ be a hypergraph with vertex partition $\{V_1,\dots, V_n\}$, $|V_i|=p_i$ for $1\leq i\leq n$ such that there are $\lambda_i$ edges of size $h_i$ incident with every $h_i$ vertices, at most one vertex from each part for $1\leq i\leq m$ (so no edge is incident with more than one vertex of a part).  
We use our detachment theorem to show that the obvious necessary conditions for $(\lambda_1\dots,\lambda_m) K^{h_1,\dots,h_m}_{p_1,\dots,p_n}$ to be  
  expressed as the union $\scr G_1\cup \ldots \cup\scr G_k$ of $k$ edge-disjoint factors, where for $1\leq i\leq k$, $\scr G_i$ is $r_i$-regular,  are also sufficient.
Baranyai \cite{Baran75, Baran79} solved the case of $h_1=\dots=h_m$, $\lambda_1=\dots,\lambda_m=1$, $p_1=\dots=p_m$, $r_1=\dots =r_k$. Berge and  Johnson  \cite{BergeJohnson75, ELJohnson78}, (and later Brouwer and Tijdeman \cite{Brouwer77hered, BrouTijd81}, respectively) considered (and solved, respectively) the case of $h_i=i$, $1\leq i\leq m$, $p_1=\dots=p_m=\lambda_1=\dots=\lambda_m=r_1=\dots =r_k=1$. We also extend our result to the case where each  $\scr G_i$ is almost regular.

In the next two sections, we give more precise definitions along with terminology. In Section \ref{mainrstate}, we state  our main result,  followed by the proof in Section \ref{proofmainthsec}. In the last section,  we show the usefulness of  the main result on decompositions of various classes of hypergraphs. We defer the applications of the main result in solving embedding problems to a  future paper.

\section{Terminology and precise definitions} 
If $x, y\in \mathbb{R}$ ($\mathbb{R}$ is the set of real numbers), then $\lfloor x \rfloor$ and $\lceil x \rceil$ denote the integers such that $x-1<\lfloor x \rfloor \leq x \leq \lceil x \rceil < x+1$, and $x\approx y$ means $\lfloor y \rfloor \leq x\leq \lceil y \rceil$. We observe that the relation $\approx$ is transitive (but not symmetric) and for $x, y \in \mathbb{R}$, and $n\in \mathbb{N}$ ($\mathbb{N}$ is the set of positive integers), $x\approx y$ implies $x/n\approx y/n$. These properties of $\approx$ will be used in Section \ref{proofmainthsec} without further explanation. For a multiset $A$ and $u\in A$, let $\mu_A (u)$ denote the multiplicity of $u$ in $A$, and let  $|A|=\sum_{u\in A} \mu_A(u)$. For multisets $A_1,\dots, A_n$, we define $A=\bigcup_{i=1}^n A_i$ by $\mu_{A}(u)=\sum_{i=1}^n \mu_{A_i}(u)$. We may use abbreviations such as $\{u^r\}$ for  $\{{\underbrace{u,\ldots,u}_r}\}$ --- 
for example $\{u^2,v,w^2\}\cup \{u,w^2\}=\{u^3,v,w^4\}$. 

For the purpose of this paper, a \textit {hypergraph} $\scr{G}$ is an ordered quintuple $(V(\scr{G}),  E(\scr{G}), H(\scr{G}),$ $\psi, \phi)$ where $V(\scr{G}),  E(\scr{G}),  H(\scr{G})$ are disjoint finite sets, $\psi:H(\scr{G}) \rightarrow V(\scr G)$ is a function and  $\phi: H(\scr G) \rightarrow E(\scr G)$ is a surjection.  
  Elements of $V(\scr{G}), E(\scr{G}), H(\scr{G})$ are called \textit{vertices}, \textit{edges} and \textit{hinges} of $\scr G$, respectively. 
A vertex $v$ (edge $e$, respectively) and hinge $h$ are said to be \textit{incident} with each other if $\psi(h)=v$ ($\phi(h)=e$, respectively). A hinge $h$ is said to \textit{attach} the edge $\phi(h)$ to the vertex $\psi(h)$. In this manner, the vertex $\psi(h)$ and the edge $\phi(h)$ are said to be \textit{incident} with each other. If $e\in E(\scr G)$, and $e$ is incident with $n$ hinges $h_1, \ldots, h_n$ for some $n\in \mathbb N$, then the edge $e$ is said to \textit{join} (not necessarily distinct) vertices $\psi(h_1), \ldots, \psi(h_n)$. 
 If $v\in V(\scr G)$, then the number of hinges incident with $v$ (i.e. $|\psi^{-1}(v)|$) is  called the \textit{degree} of $v$ and is denoted by $d(v)$.  The number of (distinct) vertices incident with an edge $e$, denoted by $|e|$, is called the \textit{size} of $e$. 
If for all edges $e$ of $\scr G$, $|e|\leq 2$ and $|\phi^{-1}(e)|=2$, then $\scr G$ is a \textit{graph}.

Thus a hypergraph, in the sense of our definition,  is a generalization of a hypergraph as it is usually defined. In fact, if for every edge $e$, $|e|=|\phi^{-1}(e)|$, then our definition is essentially the same as the usual definition. Here for convenience, we imagine each edge of a hypergraph to be attached to the vertices which it joins by in-between objects called hinges. Readers from a graph theory background may think of this as a bipartite multigraph with vertex bipartition $\{V, E\}$, in which the hinges form the edges. A hypergraph may be drawn as a set of points representing the vertices. A hyperedge is represented by a simple closed curve enclosing its incident vertices. A hinge is represented by a small line attached to the vertex incident with it (see Figure \ref{figure:hypexamplecpc}).

The set of hinges of $\scr G$ which are incident with a vertex $v$ (and an edge $e$, respectively), is denoted by $H(v)$ ($H(v,e)$, respectively).  Thus if $v\in V(\scr G)$, then $H(v)=\psi^{-1}(v)$, and $|H(v)|$ is the degree $d(v)$ of $v$. If $U$ is a multi-subset of $V(\scr G)$, and $u\in V(\scr G)$, let $E(U)$ denote the set of edges $e$ with $|\phi^{-1}(e)|=|U|$ joining vertices in $U$. More precisely, $E(U)=\{e\in E(\scr G) | \mbox { for all }v\in V(\scr G), |H(v,e)|=\mu_U(v)\}$. 
 For $U_1,\dots,U_n\subset V$ where for $1\leq i\leq n$ each $U_i$ is a multiset, let $E(U_1,\dots,U_n)$ denote $E(\bigcup_{i=1}^n U_i)$.  
We write $m(U)$ for $|E(U)|$ and call it the \textit {multiplicity} of $U$.  For simplicity, $E(u^r,U)$ denotes $E(\{u^r\},U)$, and  $m(u_1^{m_1},\dots,u_r^{m_r})$ denotes $m(\{u_1^{m_1},\dots,u_r^{m_r}\})$.  
 The set of hinges that are incident with $u$ and an edge in $E(u^r,U)$ is denoted by $H(u^r,U)$. 
\begin{example}\label{hyp1ex}
\textup{
 Let $\scr G=(V,  E, H, \psi, \phi)$, with $V=\{v_1,v_2,v_3, v_4, v_5\}, E=\{e_1,e_2, e_3\}, H=\{h_i, 1\leq i\leq 7\}$, such that $\psi(h_1)=\psi(h_2)=v_1, \psi(h_3)=v_2, \psi(h_4)=\psi(h_5)=v_3, \psi(h_6)=v_4, \psi(h_7)=v_5$ and $\phi(h_1)=\phi(h_2)=\phi(h_3)=\phi(h_4)=e_1, \phi(h_5)=\phi(h_6)=e_2, \phi(h_7)=e_3$. We have:
\begin{figure}[htbp]
\begin{center}
\scalebox{.75}{ \input{cpcnotationexample.pstex_t} }
\caption{Representation of a hypergraph $\scr G$ }
\label{figure:hypexamplecpc}
\end{center}
\end{figure} 
}\end{example}

\begin{itemize}
\item $|e_1|=3, |e_2|=2, |e_3|=1$,
\item $d(v_1)=d(v_3)=2, d(v_2)=d(v_4)=d(v_5)=1$,
\item $H(v_1)=\{h_1,h_2\}, H(v_2)=\{h_3\}, H(v_3)=\{h_4,h_5\}$,
\item $H(v_3,e_1)=\{h_4\}, H(v_3,e_2)=\{h_5\}, H(v_3,e_3)=\varnothing$,
\item $E(\{v_1,v_2,v_3\})=\varnothing, E(\{v_1^2,v_2,v_3\})=E(v_1^2,\{v_2,v_3\})=\{e_1\}$,
\item $m(v_1,v_2,v_3)=0, m(v_1^2,v_2,v_3)=1$, 
\item $H(v_1^2,\{v_2,v_3\})=\{h_1,h_2\}, H(v_1,\{v_2,v_3\})=\varnothing, H(v_3,\{v_1^2,v_2\})=\{h_4\}$.

\end{itemize} A \textit{k-edge-coloring} of $\scr G$ is a mapping $f: E(\scr G)\rightarrow C$, where $C$ is a set of $k$ \textit{colors} (often we use $C=\{1,\ldots,k\}$), and the edges of one color form a \textit{color class}. The sub-hypergraph of $\scr G$ induced by the color class $j$ is denoted by $\scr G(j)$. 
To avoid ambiguity, subscripts may be used to indicate the hypergraph in which hypergraph-theoretic notation should be interpreted --- for example, $d_{\scr G}(v)$, $E_{\scr G}(v^2,w)$, $H_{\scr G}(v)$.

\section{Amalgamations and detachments}
If $\scr F=(V,  E, H, \psi, \phi)$ is a hypergraph and $\Psi$ is a function from $V$ onto a set $W$, then we shall say that the hypergraph $\scr G=(W,  E, H, \Psi \circ \psi, \phi)$ is an \textit{amalgamation} of $\scr F$ and that $\scr F$ is a \textit{detachment} of $\scr G$. Associated with $\Psi$ is the \textit{number function} $g:W\rightarrow \mathbb N$  defined by $g(w)=|\Psi^{-1}(w)|$, for each $w\in W$; being more specific, we may also say that $\scr F$ is a \textit{g-detachment} of $\scr G$. Intuitively speaking, a $g$-detachment of $\scr G$ is obtained by splitting each $u\in V(\scr G)$ into $g(u)$ vertices. Thus $\scr F$ and $\scr G$ have the same edges and hinges, and each vertex $v$ of $\scr G$ is obtained by identifying those vertices of $\scr F$ which belong to the set $\Psi^{-1}(v)$. In this process, a hinge incident with a vertex $u$ and an edge $e$ in $\scr F$ becomes incident with the vertex $\Psi(u)$ and the  edge $e$ in $\scr G$. 

There are quite a lot of other papers on  amalgamations  and some highlights include  \cite{FerenHilton, AJWHIL80, AJWHIL87, H2, HR, MatJohns, Nash87, RW}. 

\section {Main Result} \label{mainrstate}
A function $g:V(\scr G)\rightarrow {\mathbb N}$ is said to be \textit{simple} if 
$$|H(v,e)|\leq g(v) \mbox { for } v\in V(\scr G), e\in E(\scr G).$$
A hypergraph $\scr G$ is said to be \textit{simple} if $g:V(\scr G)\rightarrow {\mathbb N}$ with $g(v)=1$ for $v\in V(\scr G)$ is simple. 
It is clear  that for a hypergraph $\scr F$ and a function $g:V(\scr F)\rightarrow {\mathbb N}$, there exists a simple $g$-detachment if and only if $g$ is simple. 
\begin{theorem} \label{mainthhypgen1}
Let $\scr F$ be a $k$-edge-colored hypergraph and let $g:V(\scr F)\rightarrow {\mathbb N}$ be a simple function. Then there exists a simple $g$-detachment $\scr G$ (possibly with multiple edges) of $\scr F$ with amalgamation function $\Psi:V(\scr G)\rightarrow V(\scr F)$,  $g$  being the number function associated with $\Psi$, such that:
\begin{itemize}
\item [\textup{(A1)}] $d_\scr G(v) \approx d_\scr F(u)/g(u) $ for each $u\in V(\scr F)$ and each $v\in \Psi^{-1}(u);$
\item [\textup{(A2)}] $d_{\scr G(j)}(v) \approx d_{\scr F(j)}(u)/g(u)$ for each $u\in V(\scr F)$, each $v\in \Psi^{-1}(u)$ and $1\leq j\leq k;$
\item [\textup{(A3)}] $m_\scr G(U_1,\dots,U_r) \approx m_\scr F(u_1^{m_1},\dots,u_r^{m_r})/\Pi_{i=1}^r\binom {g(u_i)}{m_i} $ for  distinct $u_1,\dots,u_r\in V(\scr F)$ and  $U_i\subset \Psi^{-1}(u_i)$ with $|U_i|=m_i\leq g(u_i)$ for $1\leq i\leq r;$ 
\item [\textup{(A4)}] $m_{\scr G(j)}(U_1,\dots,U_r) \approx m_{\scr F(j)}(u_1^{m_1},\dots,u_r^{m_r})/\Pi_{i=1}^r\binom {g(u_i)}{m_i} $ for  distinct $u_1,\dots,u_r\in V(\scr F)$ and  $U_i\subset \Psi^{-1}(u_i)$ with $|U_i|=m_i\leq g(u_i)$ for $1\leq i\leq r$  and $1\leq j\leq k.$
\end{itemize}
\end{theorem}
A family $\scr A$ of sets is \textit{laminar} if, for every pair $A, B$ of sets belonging to $\scr A$, either $A\subset B$, or $B\subset A$, or $A\cap B=\varnothing$. 
To prove the main result, we need the following lemma:
\begin{lemma}\textup{(Nash-Williams \cite[Lemma 2]{Nash87})}\label{laminarlem}
If $\scr A, \scr B$ are two laminar families of subsets of a finite set $S$, and $n\in \mathbb N$, then there exist a subset $A$ of $S$ such that for every $P\in \scr A \cup \scr B$,  $|A\cap P|\approx |P|/n$. 
\end{lemma}

\section{proof of Theorem \ref{mainthhypgen1}} \label{proofmainthsec}
\subsection{Inductive construction of $\scr G$ }
Let $\scr F =(V, E, H, \psi, \phi)$. Let $n = \sum\nolimits _{v \in V} (g (v) -1)$. Initially we let $\scr F_0= \scr F$ and $g_0=g$, and we let $\Phi_0$ be the identity function from $V$ into  $V$. Now assume that $\scr F_0=(V_0, E_0, H_0, \psi_0,\phi_0),\ldots,\scr F_i=(V_i,E_i, H_i,\psi_i,\phi_i)$ and $\Phi_0,\ldots,\Phi_i$ have been defined for some $i\geq 0$. Also assume that the simple functions $g_0:V_0\rightarrow\mathbb{N},\ldots, g_i:V_i\rightarrow\mathbb{N}$ have been defined for some $i\geq 0$.  Let $\Psi_i=\Phi_0\ldots\Phi_i$. If $i=n$, we terminate the construction,  letting $\scr G=\scr F_n$ and  $\Psi=\Psi_n$. 

If $i< n$, we can select a vertex $\alpha$ of $\scr F_i$ such that  $g_i(\alpha)\geq 2$. As we will see, $\scr F_{i+1}$ is formed from $\scr F_i$ by splitting off a vertex $v_{i+1}$ from $\alpha$ so that we end up with $\alpha$ and $v_{i+1}$. Let 
\begin{eqnarray}\label{lamAi}
\scr A_i & = &  \{H_{\scr F_i}(\alpha)\} \nonumber\\
&\bigcup & \{H_{\scr F_i(1)}( \alpha),\ldots,H_{\scr F_i(k)}( \alpha)\} \nonumber\\
& \bigcup & \{ H_{\scr F_i(j)}( \alpha, e) : e  \in E_{\scr F_i(j)}( \alpha), 1\leq j\leq k\}, 
\end{eqnarray}
and let 
\begin{eqnarray} \label{lamBi}
\scr B_i&=&\{H_{\scr F_i}(\alpha^t, U): t\geq 1, U \subset V_i\backslash\{\alpha\}\} \nonumber \\
&\bigcup &\{H_{\scr F_i(j)}(\alpha^t, U): t\geq 1, U \subset V_i\backslash\{\alpha\}, 1\leq j \leq k\}.
\end{eqnarray}

It is easy to see that both $\scr A_i$ and $\scr B_i$ are laminar families of subsets of $H(\scr F_i, \alpha)$. 
Therefore, by Lemma \ref{laminarlem}, there exists a subset $Z_i$ of $H(\scr F_i, \alpha)$ such that 
\begin{equation}\label{lamapp1'} |Z_i\cap P|\approx |P|/g_i(\alpha),  \mbox{ for every  } P\in \scr A_i \cup \scr B_i. \end{equation}
Let $v_{i+1}$ be a vertex which does not belong to $V_i$ and let $V_{i+1}=V_i\cup \{v_{i+1}\}$.
Let $\Phi_{i+1}$ be the function from $V_{i+1}$ onto $V_i$ such that $\Phi_{i+1}(v)=v$ for every $v\in V_i$ and $\Phi_{i+1}(v_{i+1})=\alpha$. Let $\scr F_{i+1}$ be the detachment of $\scr F_i$ under $\Phi_{i+1}$ 
such that $V(\scr F_{i+1})=V_{i+1}$, and
 \begin{equation}\label{hinge1} 
 H_{\scr F_{i+1}}(v_{i+1})=Z_i, H_{\scr F_{i+1}}(\alpha)=H_{\scr F_i}( \alpha)\backslash Z_i. 
 \end{equation}
In fact, $\scr F_{i+1}$ is obtained from $\scr F_i$ by splitting $\alpha$ into two vertices $\alpha$ and  $v_{i+1}$ in such a way that hinges which were incident with $\alpha$ in $\scr F_i$ become incident in $\scr F_{i+1}$ with $\alpha$ or $v_{i+1}$ according as they do not or do belong to $Z_i$, respectively. 
Obviously, $\Psi_i$ is an amalgamation function from $\scr F_{i+1}$ into $\scr F_i$. Let $g_{i+1}$ be the function from $V_{i+1}$ into $\mathbb N$, such that $g_{i+1}(v_{i+1})=1, g_{i+1}(\alpha)=g_i(\alpha)-1$, and $g_{i+1}(v)=g_i(v)$ for every $v\in V_i\backslash\{\alpha\}$. This finishes the construction of $\scr F_{i+1}$.   
\subsection{Relations between $\scr F_{i+1}$ and $\scr F_i$}
The hypergraph $\scr F_{i+1}$, satisfies the following conditions:
\begin{itemize}
\item [\textup{(B1)}] $d_{\scr F_{i+1}}(\alpha) \approx d_{\scr F_{i}}(\alpha)g_{i+1}(\alpha)/g_{i}(\alpha);$
\item [\textup{(B2)}] $d_{\scr F_{i+1}}(v_{i+1}) \approx d_{\scr F_{i}}(\alpha)/g_{i}(\alpha);$
\item [\textup{(B3)}] $m_{\scr F_{i+1}}(v_{i+1}^s,\alpha^t, U)=0$ for $s\geq 2$, and $t\geq 0;$
\item [\textup{(B4)}] $m_{\scr F_{i+1}}(\alpha^t, U) \approx m_{\scr F_{i}}(\alpha^t,U)(g_i(\alpha)-t)/g_i(\alpha)$ for each $U\subset V_i\backslash\{\alpha\}$, and $g_i(\alpha)\geq t\geq 1;$
\item [\textup{(B5)}] $m_{\scr F_{i+1}}(\alpha^t,v_{i+1}, U) \approx (t+1) m_{\scr F_{i}}(\alpha^{t+1}, U)/g_i(\alpha)$ for each $U\subset V_i\backslash\{\alpha\}$, and $t\geq 0.$
\end{itemize}
\begin{proof} Since $H_{\scr F_{i}}(\alpha)\in \scr A_i$, from (\ref{hinge1}) it follows that
\begin{eqnarray*}
d_{\scr F_{i+1}}(v_{i+1})& = & |H_{\scr F_{i+1}}(v_{i+1})|=|Z_i|=|Z_i\cap H_{\scr F_{i}}(\alpha)|\\
& \approx & |H_{\scr F_{i}}(\alpha)|/g_i(\alpha)=d_{\scr F_{i}}(\alpha)/g_{i}(\alpha),\\
d_{\scr F_{i+1}}(\alpha) & = &   |H_{\scr F_{i+1}}(\alpha)|= |H_{\scr F_{i}}(\alpha)|-|Z_i| \nonumber \\
& \approx&  d_{\scr F_{i}}(\alpha)-d_{\scr F_{i}}(\alpha)/g_{i}(\alpha)=(g_i(\alpha)-1)d_{\scr F_{i}}(\alpha)/g_{i}(\alpha)\nonumber \\
&= &  d_{\scr F_{i}}(\alpha)g_{i+1}(\alpha)/g_{i}(\alpha).
\end{eqnarray*}
This proves (B1) and (B2).

If $t\geq 1, U\subset V_i\backslash\{\alpha\}$, and $e\in  E_{\scr F_i}(\alpha^t, U)$, then for some $j$, $1\leq j\leq k$, $H_{\scr F_i(j)}( \alpha,e)\in \scr A_i$, so 
$$\left |Z_i\cap H_{\scr F_i(j)}( \alpha,e)\right |\approx |H_{\scr F_i(j)}( \alpha,e)|/g_i(\alpha)=t/g_i(\alpha)\leq 1,$$
where the inequality implies from the fact that $g_i$ is simple. 
 Therefore either $|Z_i\cap H_{\scr F_i(j)}( \alpha,e)|=1$ and consequently  
 $e\in E_{\scr F_{i+1}}(\alpha^{t-1}, v_{i+1}, U$) or $Z_i\cap H_{\scr F_i(j)}( \alpha,e)=\varnothing$ and consequently $e\in E _{\scr F_{i+1}}(\alpha^t, U)$. Therefore $$m_{\scr F_{i+1}}(v_{i+1}^s,\alpha^r, U)=0,$$ 
 for $r\geq 1$, and $s\geq 2$. This proves (B3).  Moreover, since $H_{\scr F_i}(\alpha^t, U)\in \scr B_i$, we have 
\begin{eqnarray*}
m_{\scr F_{i+1}}(\alpha^{t-1}, v_{i+1}, U)& = & |Z_i\cap H_{\scr F_i}(\alpha^t, U)|\approx |H_{\scr F_i}(\alpha^t, U)| /g_i(\alpha)=t m_{\scr F_{i}}(\alpha^t, U)/g_{i}(\alpha),\\
m_{\scr F_{i+1}}(\alpha^t,U) & \approx &  m_{\scr F_{i}}(\alpha^t,U)-|H_{\scr F_i}(\alpha^t, U)|/g_{i}(\alpha)=m_{\scr F_{i}}(\alpha^t,U)-t m_{\scr F_{i}}(\alpha^t,U)/g_{i}(\alpha)\nonumber \\
&= &  m_{\scr F_{i}}(\alpha^t, U)(g_{i}(\alpha)-t)/g_{i}(\alpha).
\end{eqnarray*}
This proves (B4) and (B5). 
\end{proof}
Let us fix $j\in\{1,\dots,k\}$. It is enough to replace $\scr F_i$ with $\scr F_i(j)$ in the statement and the proof of (B1)--(B5) to obtain companion conditions, say (C1)--(C5) 
for each color class.

\subsection{Relations between $\scr F_{i}$ and $\scr F$}
Recall that $\Psi_i=\Phi_0\ldots\Phi_i$, that $\Phi_0:V\rightarrow V$, and that $\Phi_i:V_i\rightarrow V_{i-1}$ for $i>0$. Therefore $\Psi_i:V_i\rightarrow V$ and thus $\Psi_i^{-1}:V\rightarrow V_i$. Now we use (B1)--(B5) to prove that the hypergraph $\scr F_{i}$ satisfies the following conditions for $0\leq i \leq n:$ 
\begin{itemize}
\item [\textup{(D1)}]  $d_{\scr F_{i}}(v)/g_{i}(v) \approx d_{\scr F}(u)/g(u)$ for each $u\in V$ and each $v\in \Psi_i^{-1}(u);$
\item [\textup{(D2)}] $m_{\scr F_{i}}(u_1^{a_1},U_1,\dots,u_r^{a_r},U_r)/\Pi_{j=1}^r\binom {g_i(u_j)}{a_j}  \approx m_\scr F(u_1^{m_1},\dots,u_r^{m_r})/\Pi_{j=1}^r\binom {g(u_j)}{m_j} $ for  distinct vertices $u_1,\dots,u_r\in V$, $a_j\geq 0$, $U_j\subset \Psi_i^{-1}(u_j)\backslash\{u_j\}$ with $1\leq m_j=a_j+|U_j| \leq g(u_j)$, $1\leq j \leq r$ if $g_i(u_j)\geq a_j$, $1\leq j \leq r$. 
\end{itemize}
\begin{proof} The proof is by induction. Recall that $\scr F_0=\scr F$, and $g_0(u)= g(u)$ for $u\in V$. Thus, (D1) and (D2) are trivial for $i=0$. 
Now we will show that if $\scr F_i$ satisfies the conditions (D1) and (D2) for some $i< n$, then $\scr F_{i+1}$  satisfies these conditions by replacing $i$ with $i+1$; we denote the corresponding conditions for $\scr F_{i+1}$ by (D1)$'$ and (D2)$'$. 

Let $u\in V$. If $g_{i+1}(u)=g_i(u)$, then (D1)$'$ is obviously true. So we just check (D1)$'$ in the case where $u=\alpha$. By (B1) and (D1) we have $d_{\scr F_{i+1}}(\alpha)/g_{i+1}(\alpha) \approx d_{\scr F_{i}}(\alpha)/g_{i}(\alpha)\approx d_{\scr F}(\alpha)/g(\alpha).$
 Moreover, from (B2) and (D1) it follows that $d_{\scr F_{i+1}}(v_{i+1}) \approx d_{\scr F_{i}}(\alpha)/g_{i}(\alpha)\approx d_{\scr F}(\alpha)/g(\alpha).$
 Since in forming $\scr F_{i+1}$ no edge is detached from $v_r$ for each $v_r \in \Psi_i^{-1}(\alpha)\backslash\{\alpha\}$, we have $d_{\scr F_{i+1}}(v_{r}) =d_{\scr F_{i}}(v_{r})$. Therefore $d_{\scr F_{i+1}}(v_{r})=d_{\scr F_{i}}(v_{r})\approx d_{\scr F}(\alpha)/g(\alpha)$
for each $v_r \in \Psi_i^{-1}(\alpha)\backslash\{\alpha\}$. This proves (D1)$'$. 
 Let $u_1,\dots,u_r$ be distinct vertices in $V$. If $g_{i+1}(u_j)=g_i(u_j)$ for $1\leq j\leq r$, then (D2)$'$ is clearly true. Therefore, in order to prove (D2)$'$, without loss of generality we may assume that $g_{i+1}(u_1)=g_i(u_1)-1$ (so $\alpha=u_1$  and $v_{i+1}\in \Psi_i^{-1}(u_1)$). First, note that for  integers $a, b$ we always have $(a-b)\binom{a}{b}=a\binom{a-1}{b} =(b+1)\binom{a}{b+1}$. 
If $v_{i+1}\notin U_1$, we have
\begin{eqnarray*}
\frac{m_{\scr F_{i+1}}(u_1^{a_1},U_1,\dots,u_r^{a_r},U_r)}{\Pi_{j=1}^r\binom {g_{i+1}(u_j)}{a_j}} & \mathop\approx \limits^{\textup{(B4)}} &\frac{m_{\scr F_{i}}(u_1^{a_1},U_1,\dots,u_r^{a_r},U_r)(g_i(u_1)-a_1)/g_i(u_1)}{\binom{g_i(u_1)-1}{a_1}\Pi_{j=2}^r\binom {g_i(u_j)}{a_j}}\\
&=&\frac{m_{\scr F_{i}}(u_1^{a_1},U_1,\dots,u_r^{a_r},U_r)(g_i(u_1)-a_1)/g_i(u_1)}{(g_i(u_1)-a_1)/g_i(u_1)\binom{g_i(u_1)}{a_1}\Pi_{j=2}^r\binom {g_i(u_j)}{a_j}}\\
&=&\frac{m_{\scr F_{i}}(u_1^{a_1},U_1,\dots,u_r^{a_r},U_r)}{\Pi_{j=1}^r\binom {g_i(u_j)}{a_j}}\\
&\mathop \approx \limits^{\textup{(D2)}}&   \frac{m_\scr F(u_1^{m_1},\dots,u_r^{m_r})}{\Pi_{j=1}^r\binom {g(u_j)}{m_j}}.
\end{eqnarray*}
If $v_{i+1}\in U_1$, we have
\begin{eqnarray*}
\frac{m_{\scr F_{i+1}}(u_1^{a_1},U_1,\dots,u_r^{a_r},U_r)}{\Pi_{j=1}^r\binom {g_{i+1}(u_j)}{a_j}} & \mathop\approx \limits^{\textup{(B5)}} & \frac{m_{\scr F_{i}}(u_1^{a_1+1},U_1\backslash\{v_{i+1}\},\dots,u_r^{a_r},U_r)(a_1+1)/g_i(u_1)}{\binom{g_i(u_1)-1}{a_1}\Pi_{j=2}^r\binom {g_i(u_j)}{a_j}}\\
&=&\frac{m_{\scr F_{i}}(u_1^{a_1+1},U_1\backslash\{v_{i+1}\},\dots,u_r^{a_r},U_r)}{g_i(u_1)/(a_1+1)\binom{g_i(u_1)-1}{a_1}\Pi_{j=2}^r\binom {g_i(u_j)}{a_j}}\\
&=&\frac{m_{\scr F_{i}}(u_1^{a_1+1},U_1\backslash\{v_{i+1}\},\dots,u_r^{a_r},U_r)}{\binom{g_i(u_1)}{a_1+1}\Pi_{j=2}^r\binom {g_i(u_j)}{a_j}}\\
&\mathop\approx \limits^{\textup{(D2)}}& \frac{m_\scr F(u_1^{m_1},\dots,u_r^{m_r})}{\Pi_{j=1}^r\binom {g(u_j)}{m_j}}.
\end{eqnarray*}
This proves (D2)$'$. 
\end{proof}
Let us fix $j\in\{1,\dots,k\}$. It is enough to replace $\scr F$ with $\scr F(j)$, $\scr F_i$ with $\scr F_i(j)$, $\scr F_{i+1}$ with $\scr F_{i+1}(j)$, and (B$i$) with (C$i$) for  $i=1,2,4,5$, in the statement and the proof of (D1) and (D2) to obtain companion conditions, say (E1) and (E2)  for each color class. 
\subsection{$\scr G$ satisfies (A1)--(A4) }
Recall that $\scr G=\scr F_n$ and  $g_n(u)=1$ for every $u\in V$, therefore when $i=n$, (D1) implies (A1). Moreover, if we let $i=n$ in (D2), we have $a_j\in \{0,1\}$ for $1\leq j\leq r$ and thus $\Pi_{j=1}^r\binom {g_i(u_j)}{a_j}=\Pi_{j=1}^r\binom {1}{a_j}=1$. This proves (A3). By a similar argument, one can prove (A2) and (A4), and this completes the proof of Theorem \ref{mainthhypgen1}. \qed

\section{corollaries} \label{corrfacgen}
For a matrix $A$, 
 let $A_j$ denote the $j^{th}$ column  of $A$, and let $s(A)$ denote the sum of all the elements of $A$.  
Let $R=[r_1\dots r_k]^T$ (or $R^T=[r_i]_{1\times k}$), $\Lambda=[\lambda_1\dots\lambda_m]^T$ and $H=[h_1\dots h_m]^T$ be three column vectors with $r_i,\lambda_i\in {\mathbb N}$, and $h_i\in \{1,\dots,n\}$ for $1\leq i\leq m$, such that $h_1\dots, h_m$ are distinct.  
Let $\Lambda K_n^H$ denote a hypergraph with vertex set $V$, $|V|=n$, such that there are $\lambda_i$ edges of size $h_i$ incident with every $h_i$ vertices  for $1\leq i\leq m$.
A hypergraph $\scr G$ 
is said to be \textit{$k$-regular}  if every vertex has degree $k$. A \textit{$k$-factor} of $\scr G$ is a $k$-regular spanning sub-hypergraph of $\scr G$. 
 An \textit{$R$-factorization} is a partition (decomposition) $\{F_1,\ldots, F_k\}$ of $E(\scr G)$ in which $F_i$ is an $r_i$-factor for $1\leq i\leq  k$. 
 Notice that $\Lambda K_n^H$ is $\sum_{i=1}^m  \lambda_i\binom{n-1}{h_i-1}$-regular. 
We show that the obvious necessary conditions for the existence of an $R$-factorization of $\Lambda K_n^H$, are also sufficient.

\begin{theorem} \label{r1rkfacknhlam}
$\Lambda K_n^H$ is $R$-factorizable if and only if  
$s(R)=\sum_{i=1}^m  \lambda_i\binom{n-1}{h_i-1}$, 
and there exists a non-negative integer matrix $A=[a_{ij}]_{k \times m}$ such that $AH=nR$, and $s(A_j)=\lambda_j\binom{n}{h_j}$ for $1\leq j\leq m.$
\end{theorem}
\begin{proof}
To prove the necessity, suppose that $\Lambda K_n^H$ is $R$-factorizable. Since each $r_i$-factor is an $r_i$-regular spanning sub-hypergraph for $1\leq i\leq k$, and $\Lambda K_n^H$ is $\sum_{i=1}^m \lambda_i \binom{n-1}{h_i-1}$-regular, we must have $s(R)=\sum_{i=1}^k r_i=\sum_{i=1}^m \lambda_i \binom{n-1}{h_i-1}$. Let $a_{ij}$ be  the number of edges (counting multiplicities) of size $h_j$ contributing to the $i^{th}$ factor for $1\leq i\leq k$, $1\leq j\leq m$. 
Since   for $1\leq j\leq m$, each edge of size $h_j$ contributes $h_j$ to the the sum of the degrees of the vertices in an $r_i$-factor for $1\leq i\leq k$,  we must have  $\sum_{j=1}^m a_{ij} h_j=nr_i$ for  $1\leq i\leq k$ and $\sum_{i=1}^k a_{ij} =\lambda_j\binom{n}{h_j}$  for $1\leq j\leq m$. 

To prove the sufficiency, let $\scr F$ be a hypergraph consisting of a single vertex $v$ 
with  $m_{\scr F}(v^{h_j})=\lambda_j \binom{n}{h_j}$ for $1\leq j\leq m$. Note that $\scr F$ is an amalgamation of $\Lambda K_n^H$. 
Now we color the edges of $\scr F$ so that $m_{\scr F(i)}(v^{h_j})=a_{ij}$ for $1\leq i \leq k$, $1\leq j \leq m$. This can be done, because:
\begin{equation*}
\sum_{i=1}^k m_{\scr F(i)}(v^{h_j}) =\sum_{i=1}^k a_{ij}= 
\lambda_j\binom{n}{h_j}=m_{\scr F}(v^{h_j}) \mbox {\quad for } 1\leq j\leq m. 
\end{equation*}
Moreover,  
$$d_{\scr F(i)}(v)=\sum_{j=1}^m a_{ij} h_j=nr_i                        \mbox{ \quad for } 1\leq i \leq k.
$$ 
Let $g:V(\scr F)\rightarrow \mathbb N$ be a function so that $g(v)=n$. Since for $1\leq i\leq m$, $h_i\leq n$,  $g$ is simple.   
By Theorem \ref{mainthhypgen1}, there exists a simple $g$-detachment $\scr G$ of $\scr F$ with $n$ vertices, say $v_1,\ldots, v_n$ such that by (A2), $d_{\scr G(i)}(v_j) \approx d_{\scr F(i)}(v)/g(v)=nr_i/n=r_i$ for $1\leq i\leq k$, $1\leq j\leq n$, and by (A3), for each $U\subset \{v_1,\dots,v_n\}$ with $|U|=h_j$, $m_\scr G(U) \approx m_\scr F(v^{h_j})/\binom {n}{h_j} =\lambda_j\binom{n}{h_j}/\binom {n}{h_j}=\lambda_j$ for $1\leq j\leq m$. Therefore $\scr G\cong  \Lambda K_n^H$, and the $i^{th}$ color class induces an $r_i$-factor for $1\leq i\leq k$. 
\end{proof}
In particular, if $m=1$, $h:=h_1$, $\lambda_1=1$,  $r:=r_1=\dots =r_k$, then Theorem \ref{r1rkfacknhlam} implies Baranyai's theorem: the complete $h$-uniform hypergraph $K_n^h$ is $r$-factorizable if and only if $h\divides rn$ and $r\divides \binom{n-1}{h-1}$.

Now let $h_i\geq 2$ for $1\leq i\leq m$, and  
let $\Lambda K^H_{p_1,\dots,p_n}$ be a hypergraph with vertex partition $\{V_1,\dots, V_n\}$, $|V_i|=p_i$ for $1\leq i\leq n$ such that there are $\lambda_i$ edges of size $h_i$ incident with every $h_i$ vertices, at most one vertex from each part for $1\leq i\leq m$ (so no edge is incident with more than one vertex of a part). If $p_1=\dots=p_n:=p$, we denote $\Lambda K^H_{p_1,\dots,p_n}$ by $\Lambda K^H_{n\times p}$. 
 \begin{theorem} \label{r1rkfacknphlam}
 $\Lambda K_{p_1,\dots, p_n}^{H}$ is $R$-factorizable if and only if $p_1=\dots=p_n:=p$, 
$s(R)=\sum_{i=1}^m  \lambda_i\binom{n-1}{h_i-1}p^{h_i-1}$, 
and there exists a non-negative integer matrix $A=[a_{ij}]_{k \times m}$ such that $AH=npR$, and $s(A_j)=\lambda_j\binom{n}{h_j}p^{h_j}$ for $1\leq j\leq m.$
\end{theorem}
\begin{proof}
To prove the necessity, suppose that $\Lambda K_{p_1,\dots, p_n}^{H}$ is $R$-factorizable (so it is regular). Let $u$ and $v$ be two vertices from two different parts, say $a^{th}$ and $b^{th}$ parts, respectively. Since $d(u)=d(v)$, we have 
\begin{equation*}\label{regularppart} \begin{split}
\sum_{1\leq j\leq m}\lambda_j \sum\nolimits_{\scriptstyle 1\leq i_1<\dots<i_{_{h_j-1}}\leq n \hfill \atop  \scriptstyle a\notin\{ i_1,\dots,i_{_{h_j-1}}\} \hfill}  p_{i_1}\dots p_{i_{h_j-1}}=   
\sum_{1\leq j\leq m}\lambda_j \sum\nolimits_{\scriptstyle 1\leq i_1<\dots<i_{_{h_j-1}}\leq n \hfill \atop  \scriptstyle b\notin\{ i_1,\dots,i_{_{h_j-1}}\} \hfill}  p_{i_1}\dots p_{i_{h_j-1}}
&\iff \\
\sum_{1\leq j\leq m}\lambda_j \Big( \sum\nolimits_{\scriptstyle 1\leq i_1<\dots<i_{_{h_j-1}}\leq n \hfill \atop  \scriptstyle a\notin\{ i_1,\dots,i_{_{h_j-1}}\} \hfill}  p_{i_1}\dots p_{i_{h_j-1}}-  \sum\nolimits_{\scriptstyle 1\leq i_1<\dots<i_{_{h_j-1}}\leq n \hfill \atop  \scriptstyle b\notin\{ i_1,\dots,i_{_{h_j-1}}\} \hfill}  p_{i_1}\dots p_{i_{h_j-1}}\Big)=0
&\iff \\
\sum_{1\leq j\leq m}\lambda_j \Big(p_b\sum\nolimits_{\scriptstyle 1\leq i_1<\dots<i_{_{h_j-2}}\leq n}  p_{i_1}\dots p_{i_{h_j-2}}-  p_a\sum\nolimits_{\scriptstyle 1\leq i_1<\dots<i_{_{h_j-2}}\leq n}  p_{i_1}\dots p_{i_{h_j-2}}\Big)=0
&\iff \\
(p_b-p_a)\sum_{1\leq j\leq m}\lambda_j \sum\nolimits_{\scriptstyle 1\leq i_1<\dots<i_{_{h_j-2}}\leq n}  p_{i_1}\dots p_{i_{h_j-2}}
= 
0 &\iff \\
p_b=p_a.&
 \end{split} \end{equation*}
Therefore, $p_1=\dots=p_n:=p$.  So $\Lambda K^H_{n\times p}$ is $\sum_{i=1}^m  \lambda_i\binom{n-1}{h_i-1}p^{h_i-1}$-regular, and we  must have $s(R)=\sum_{i=1}^k r_i=\sum_{i=1}^m  \lambda_i\binom{n-1}{h_i-1}p^{h_i-1}$. 
Moreover,  there must exist non-negative integers $a_{ij}$, $1\leq i\leq k$, $1\leq j\leq m$, such that $\sum_{j=1}^m a_{ij} h_j=npr_i$ for $1\leq i\leq k$ and $\sum_{i=1}^k a_{ij} =\lambda_j\binom{n}{h_j}p^{h_j}$  for $1\leq j\leq m$. We note that $a_{ij}$ is in fact the number of edges (counting multiplicities) of size $h_j$ contributing to the $i^{th}$ factor. 

To prove the sufficiency, let $\Lambda^p=[p^{h_i}\lambda_i]_{1\times m}^T$, and let $\scr F=\Lambda^p K^H_{n}$ with vertex set $V=\{v_1,\dots,v_n\}$. 
Notice  that $\scr F$ is an amalgamation of $\Lambda K^H_{n\times p}$. By Theorem \ref{r1rkfacknhlam}, $\scr F$ is  $pR$-factorizable. 
Therefore, we can color the edges of $\scr F$ so that 
$$d_{\scr F(i)}(v)=pr_i \mbox{ for } v\in V, 1\leq i \leq k.$$ 
Let $g:V\rightarrow \mathbb N$ be a function so that $g(v)=p$ for $v\in V$. Since $p\geq 1$,  $g$ is simple.   
By Theorem \ref{mainthhypgen1}, there exists a simple $g$-detachment $\scr G$ of $\scr F$ with $np$ vertices, say $v_i$ is detached to $v_{i1},\dots, v_{ip}$ for $1\leq i\leq n$,  such that by (A2), $d_{\scr G(i)}(v_{ab}) \approx d_{\scr F(i)}(v_a)/g(v_a)=pr_i/p=r_i$ for $1\leq i\leq k$, $1\leq a\leq n$, $1\leq b\leq p$, and by (A3), $m_\scr G(v_{a_1b_1},\dots, v_{a_{h_j}b_{h_j}}) \approx m_\scr F(v_{a_1},\dots, v_{a_{h_j}})/p^{h_j} =p^{h_j}\lambda_j/p^{h_j}=\lambda_j$ for $1\leq j\leq m$, $1\leq a_1<\dots<a_{h_j}\leq n$, $1\leq b_1,\dots,b_{h_j}\leq p$.  Therefore $\scr G\cong  \Lambda K^H_{n\times p}$, and the $i^{th}$ color class induces an $r_i$-factor for $1\leq i\leq k$. 
\end{proof}
In particular, if $m=1$, $h:=h_1$, $\lambda_1=1$,  $r:=r_1=\dots =r_k$, then Theorem \ref{r1rkfacknphlam} implies another one of Baranyai's theorems: the complete $h$-uniform $n$-partite hypergraph $K_{n\times p}^h$ is $r$-factorizable if and only if $h\divides npr$ and $r\divides \binom{n-1}{h-1}p^{h-1}$.

Let $J_k^T=[1\dots1]_{1\times k}$. For two column vectors $Q=[q_1\dots q_k]^T$, $R=[r_1\dots r_k]^T$, if $q_i\leq r_i$ for $1\leq i\leq k$, we say that  $Q\leq R$. 
For a  hypergraph $\scr G$, a $(q,r)$-factor is a spanning sub-hypergraph in which 
$$
q\leq d(v) \leq r \mbox { for each } v\in V(\scr G).
$$
A \textit{$(Q,R)$-factorization} is a partition $\{F_1,\ldots, F_k\}$ of $E(\scr G)$ in which $F_i$ is a $(q_i,r_i)$-factor for $1\leq i\leq  k$.  
An \textit{almost $k$-factor} of $\scr G$ is $(k-1,k)$-factor.  An \textit{almost $R$-factorization} is an $(R-J_k,R)$-factorization. 
The proof of the following  theorems are very similar to those of Theorem \ref{r1rkfacknhlam} and  \ref{r1rkfacknphlam}.  

\begin{theorem} \label{qrfacknhlam}
$\Lambda K_n^H$ is  $(Q,R)$-factorizable  if and only if  $s(Q)\leq \sum_{i=1}^m  \lambda_i\binom{n-1}{h_i-1}\leq s(R)$, and there exists a non-negative integer matrix $A=[a_{ij}]_{k \times m}$ such that $nQ\leq AH\leq nR$, and $s(A_j)=\lambda_j\binom{n}{h_j}$ for $1\leq j\leq m$. 
\end{theorem}

\begin{proof}
To prove the necessity, suppose that $\Lambda K_n^H$ is $(Q,R)$-factorizable. Since  $\Lambda K_n^H$ is $\sum_{i=1}^m \lambda_i \binom{n-1}{h_i-1}$-regular, we must have $s(Q)=\sum_{i=1}^k q_i\leq \sum_{i=1}^m \lambda_i \binom{n-1}{h_i-1}\leq \sum_{i=1}^k r_i=s(R)$. Since   for $1\leq j\leq m$, each edge of size $h_j$ contributes $h_j$ to the the sum of the degrees of the vertices in   $(q_i,r_i)$-factor for $1\leq i\leq k$,  there must exist non-negative integers $a_{ij}$, $1\leq i\leq k$, $1\leq j\leq m$, such that $nq_i\leq \sum_{j=1}^m a_{ij} h_j\leq nr_i$ for  $1\leq i\leq k$ and $\sum_{i=1}^k a_{ij} =\lambda_j\binom{n}{h_j}$  for $1\leq j\leq m$. 

To prove the sufficiency, let $\scr F$ be a hypergraph consisting of a single vertex $v$  
with  $m_{\scr F}(v^{h_j})=\lambda_j \binom{n}{h_j}$ for $1\leq j\leq m$. Note that $\scr F$ is an amalgamation of $\Lambda K_n^H$. 
Now we color the edges of $\scr F$ so that $m_{\scr F(i)}(v^{h_j})=a_{ij}$ for $1\leq i \leq k$, $1\leq j \leq m$. This can be done, because:
\begin{equation*}
\sum_{i=1}^k m_{\scr F(i)}(v^{h_j}) =\sum_{i=1}^k a_{ij}= 
\lambda_j\binom{n}{h_j}=m_{\scr F}(v^{h_j}) \mbox {\quad for } 1\leq j\leq m. 
\end{equation*}
Moreover,  
$$nq_i\leq d_{\scr F(i)}(v)=\sum_{j=1}^m a_{ij} h_j\leq nr_i                        \mbox{ \quad for } 1\leq i \leq k.
$$ 
Let $g:V(\scr F)\rightarrow \mathbb N$ be a function so that $g(v)=n$. Since for $1\leq i\leq m$, $h_i\leq n$,  $g$ is simple.   
By Theorem \ref{mainthhypgen1}, there exists a simple $g$-detachment $\scr G$ of $\scr F$ with $n$ vertices, say $v_1,\ldots, v_n$ such that by (A2), $q_i=nq_i/n\leq d_{\scr G(i)}(v_j) \leq nr_i/n=r_i$ for $1\leq i\leq k$, $1\leq j\leq n$, and by (A3), for each $U\subset \{v_1,\dots,v_n\}$ with $|U|=h_j$, $m_\scr G(U) \approx m_\scr F(v^{h_j})/\binom {n}{h_j} =\lambda_j\binom{n}{h_j}/\binom {n}{h_j}=\lambda_j$ for $1\leq j\leq m$. Therefore $\scr G\cong  \Lambda K_n^H$, and the $i^{th}$ color class induces a $(q_i,r_i)$-factor for $1\leq i\leq k$. 
\end{proof}

\begin{theorem} \label{almostr1rkfacknhlam}
$\Lambda K_n^H$ is almost $R$-factorizable  if and only if  $s(R)-k\leq \sum_{i=1}^m  \lambda_i\binom{n-1}{h_i-1}\leq s(R)$, and there exists a non-negative integer matrix $A=[a_{ij}]_{k \times m}$ such that $n(R-J_k)\leq AH\leq nR$, and $s(A_j)=\lambda_j\binom{n}{h_j}$ for $1\leq j\leq m$. 
\end{theorem}
\begin{proof}
It is enough to take $Q=R-J_k$ in Theorem \ref{qrfacknhlam}. 
\end{proof}

\begin{theorem} \label{qrfacknhlammulti}
$\Lambda K_{n\times p}^{H}$ is  $(Q,R)$-factorizable  if and only if 
$s(Q)\leq \sum_{i=1}^m  \lambda_i\binom{n-1}{h_i-1}p^{h_i-1}\leq s(R)$, and there exists a non-negative integer matrix $A=[a_{ij}]_{k \times m}$ such that  
$npQ\leq AH\leq npR$, and $s(A_j)=\lambda_j\binom{n}{h_j}p^{h_j}$ for $1\leq j\leq m.$
\end{theorem}
\begin{proof}
To prove the necessity, suppose that $\Lambda K_{n\times p}^{H}$ is  $(Q,R)$-factorizable.  Since  $\Lambda K_{n\times p}^{H}$ is $\sum_{i=1}^m  \lambda_i\binom{n-1}{h_i-1}p^{h_i-1}$-regular, we must have $s(Q)=\sum_{i=1}^k q_i\leq \sum_{i=1}^m  \lambda_i\binom{n-1}{h_i-1}p^{h_i-1}\leq \sum_{i=1}^k r_i=s(R)$. 
Moreover,  there must exist non-negative integers $a_{ij}$, $1\leq i\leq k$, $1\leq j\leq m$, such that $npq_i\leq\sum_{j=1}^m a_{ij} h_j\leq npr_i$ for $1\leq i\leq k$ and $\sum_{i=1}^k a_{ij} =\lambda_j\binom{n}{h_j}p^{h_j}$  for $1\leq j\leq m$. 

To prove the sufficiency, let $\Lambda^p=[p^{h_i}\lambda_i]_{1\times m}^T$, and let $\scr F=\Lambda^p K^H_{n}$ with vertex set $V=\{v_1,\dots,v_n\}$. 
Notice  that $\scr F$ is an amalgamation of $\Lambda K^H_{n\times p}$.  
By Theorem \ref{qrfacknhlam}, $\scr F$ is  $(pQ,pR)$-factorizable. 
Therefore, we can color the edges of $\scr F$ so that 
$$pq_i\leq d_{\scr F(i)}(v)\leq pr_i \mbox{ for } v\in V, 1\leq i \leq k.$$ 
Let $g:V\rightarrow \mathbb N$ be a function so that $g(v)=p$ for $v\in V$. Since $p\geq 1$,  $g$ is simple.   
By Theorem \ref{mainthhypgen1}, there exists a simple $g$-detachment $\scr G$ of $\scr F$ with $np$ vertices, say $v_i$ is detached to $v_{i1},\dots, v_{ip}$ for $1\leq i\leq n$,  such that by (A2), $q_i=pq_i/p\leq d_{\scr G(i)}(v_{ab}) \leq pr_i/p=r_i$ for $1\leq i\leq k$, $1\leq a\leq n$, $1\leq b\leq p$, and by (A3), $m_\scr G(v_{a_1b_1},\dots, v_{a_{h_j}b_{h_j}}) \approx m_\scr F(v_{a_1},\dots, v_{a_{h_j}})/p^{h_j} =p^{h_j}\lambda_j/p^{h_j}=\lambda_j$ for $1\leq j\leq m$, $1\leq a_1<\dots<a_{h_j}\leq n$, $1\leq b_1,\dots,b_{h_j}\leq p$.  Therefore $\scr G\cong  \Lambda K^H_{n\times p}$, and the $i^{th}$ color class induces a $(p_i,r_i)$-factor for $1\leq i\leq k$. 
\end{proof}

\begin{theorem} 
$\Lambda K_{n\times p}^{H}$ is almost $R$-factorizable  if and only if 
$s(R)-k\leq \sum_{i=1}^m  \lambda_i\binom{n-1}{h_i-1}p^{h_i-1}\leq s(R)$, and there exists a non-negative integer matrix $A=[a_{ij}]_{k \times m}$ such that  
$np(R-J_k)\leq AH\leq npR$, and $s(A_j)=\lambda_j\binom{n}{h_j}p^{h_j}$ for $1\leq j\leq m.$
\end{theorem}
\begin{proof}
It is enough to take $Q=R-J_k$ in Theorem \ref{qrfacknhlammulti}. 
\end{proof}

\section{Acknowledgment}
The author wishes to thank the  referee and professor D. G. Hoffman for very carefully reading this manuscript and many suggestions.

\bibliographystyle{model1-num-names}
\bibliography{<your-bib-database>}

\end{document}